\documentclass[peerreview,11pt]{IEEEtran}
\IEEEoverridecommandlockouts
\usepackage[utf8]{inputenc}
\usepackage[TS1,T1]{fontenc}
\usepackage{fourier, heuristica}
\usepackage{array, booktabs}
\usepackage{graphicx}
\usepackage[x11names]{xcolor}
\usepackage{colortbl}
\usepackage{caption}
\DeclareCaptionFont{blue}{\color{LightSteelBlue3}}
\usepackage[T1]{fontenc}
\usepackage{float}
\usepackage{amsthm}
\usepackage{amsmath}
\usepackage{amssymb}
\usepackage{algpseudocode}
\usepackage{algorithm}
\usepackage{graphicx}
\usepackage{cite}
\usepackage[unicode=true,
 bookmarks=true,bookmarksnumbered=true,bookmarksopen=true,bookmarksopenlevel=1,
 breaklinks=false,pdfborder={0 0 0},backref=false,colorlinks=false]
 {hyperref}
\hypersetup{pdftitle={Your Title},
 pdfauthor={Your Name},
 pdfpagelayout=OneColumn,pdfnewwindow=true,pdfstartview=XYZ,plainpages=false}

\theoremstyle{plain}

\theoremstyle{definition}

\newtheorem{remark}{Remark}
\theoremstyle{plain}

\newtheorem{theorem}{Theorem}[section]

\newtheorem{lemma}[theorem]{Lemma}

\newtheorem{corollary}[theorem]{Corollary}
\theoremstyle{definition}

\providecommand{\corollaryname}{Corollary}
\providecommand{\propname}{Proposition}
\providecommand{\definitionname}{Definition}
\providecommand{\theoremname}{Theorem}
\providecommand{\assumptionname}{Assumption}

\makeatother

\providecommand{\corollaryname}{Corollary}
\providecommand{\definitionname}{Definition}
\providecommand{\theoremname}{Theorem}

\newcommand{\declarecommand}[1]{\providecommand{#1}{}\renewcommand{#1}}

\declarecommand{\A}{\mathbb{A}}
\declarecommand{\B}{\mathbb{B}}
\declarecommand{\C}{\mathbb{C}}
\declarecommand{\D}{\mathbb{D}}
\declarecommand{\E}{\mathbb{E}}
\declarecommand{\F}{\mathbb{F}}
\declarecommand{\G}{\mathbb{G}}
\declarecommand{\H}{\mathbb{H}}
\declarecommand{\I}{\mathbb{I}}
\declarecommand{\J}{\mathbb{J}}
\declarecommand{\K}{\mathbb{K}}
\declarecommand{\L}{\mathbb{L}}
\declarecommand{\M}{\mathbb{M}}
\declarecommand{\N}{\mathbb{N}}
\declarecommand{\O}{\mathbb{O}}
\declarecommand{\P}{\mathbb{P}}
\declarecommand{\Q}{\mathbb{Q}}
\declarecommand{\R}{\mathbb{R}}
\declarecommand{\S}{\mathbb{S}}
\declarecommand{\T}{\mathbb{T}}
\declarecommand{\U}{\mathbb{U}}
\declarecommand{\V}{\mathbb{V}}
\declarecommand{\W}{{{W_{\tau,x}}}}
\declarecommand{\X}{\mathbb{X}}
\declarecommand{\Y}{\mathbb{Y}}
\declarecommand{\Z}{\mathbb{Z}}
\declarecommand{\SSS}{{{S}_{\tau, x}^{\eta,y}}}
\declarecommand{\Ac}{\mathcal{A}}
\declarecommand{\Bc}{\mathcal{B}}
\declarecommand{\Cc}{\mathcal{C}}
\declarecommand{\Dc}{\mathcal{D}}
\declarecommand{\Ec}{\mathcal{E}}
\declarecommand{\Fc}{\mathcal{F}}
\declarecommand{\Gc}{\mathcal{G}}
\declarecommand{\Hc}{\mathcal{H}}
\declarecommand{\Ic}{\mathcal{I}}
\declarecommand{\Jc}{\mathcal{J}}
\declarecommand{\Kc}{\mathcal{K}}
\declarecommand{\Lc}{\mathcal{L}}
\declarecommand{\Mc}{\mathcal{M}}
\declarecommand{\Nc}{\mathcal{N}}
\declarecommand{\Oc}{\mathcal{O}}
\declarecommand{\Pc}{\mathcal{P}}
\declarecommand{\Qc}{\mathcal{Q}}
\declarecommand{\Rc}{\mathcal{R}}
\declarecommand{\Sc}{\mathcal{S}}
\declarecommand{\Tc}{\mathcal{T}}
\declarecommand{\Uc}{\mathcal{U}}
\declarecommand{\Vc}{\mathcal{V}}
\declarecommand{\Wc}{\mathcal{W}}
\declarecommand{\Xc}{\mathcal{X}}
\declarecommand{\Yc}{\mathcal{Y}}
\declarecommand{\Zc}{\mathcal{Z}}

\newcommand{\1}{\mathbb{1}}

\newcommand{\eps}{\varepsilon}

\DeclareMathOperator{\erf}{erf}
\DeclareMathOperator{\sign}{sign}

\title{On the Number of Faces and Radii of Cells Induced by Gaussian Spherical Tessellations
\thanks{A.M. was partially supported by U.S. Air Force Award FA9550-18-1-0031 led by Roman Vershynin. R.S. was supported in part by NSF DMS-1517204 and DMS-2012546.}
}

\author{
    \IEEEauthorblockN{Eric Lybrand (elybrand@ucsd.edu)\\}
    \IEEEauthorblockA{Department of Mathematics, University of California, San Diego\\
    }\medskip
    \and
    \IEEEauthorblockN{Anna Ma (anna.ma@uci.edu)\\}
    \IEEEauthorblockA{Department of Mathematics, University of California, Irvine\\
    }
    \and \medskip
    \IEEEauthorblockN{Rayan Saab (rsaab@ucsd.edu)\\}
    \IEEEauthorblockA{Department of Mathematics and Halicioglu Data Science Institute, \\ University of California, San Diego
    }
}

\begin{document}

\maketitle
\begin{abstract}
We study a geometric property related to spherical hyperplane tessellations in \(\R^{d}\). We first consider a fixed $x$ on the Euclidean sphere and tessellations with \(M \gg d\) hyperplanes passing through the origin having normal vectors distributed according to a Gaussian distribution. We show that with high probability there exists a subset of the hyperplanes whose cardinality is on the order of \(d\log(d)\log(M)\) such that the radius of the cell containing \(x\) induced by these hyperplanes is bounded above by, up to constants, \(d\log(d)\log(M)/M\). We extend this result to hold for all cells in the tessellation with high probability. Up to logarithmic terms, this upper bound matches the  previously established lower bound of Goyal et al. (IEEE T. Inform. Theory 44(1):16–31, 1998).
\end{abstract}

\section{Introduction}
\noindent
Suppose we have a collection of non-zero vectors \(\varphi^{(1)}, \hdots, \varphi^{(M)} \in \R^{d}\) and we consider the hyperplane tessellation these vectors induce on the unit sphere \(S^{d-1}\). One might be tempted to believe that the number of cells formed by the hyperplane tessellation --that is, for \(x \in S^{d-1}\) with \(\langle \varphi^{(i)}, x\rangle \neq 0\) for all \(i \in [M]\), sets of the form \(\{y \in S^{d-1} : \sign(\langle x, \varphi^{(i)} \rangle) = \sign(\langle y, \varphi^{(i)} \rangle) \,\, \text{ for all } i \in [M]\}\)-- is exponential in \(M\). However, it is not difficult to see that this is far from the truth when \(M\) is much larger than \(d\), i.e., $M\gg d$. In this setting and when the \(\varphi^{(i)}\) are in general position, or when any \(d\) of them are linearly independent, Schl\"afli proved, in the 1800s, that the number of cells is exactly
\begin{align}\label{eq: number of cells}
    2\sum_{i=0}^{d-1}\binom{M-1}{i}.
\end{align}
See \cite{schneider2008stochastic}  for a version of Schl\"afi's proof and for further references. To tease out what the asymptotic properties of this quantity are in terms of $M$ and $d$, note that for any
\(\frac{M+1}{2} > d\) the leading term in the sum is the final summand. Notice that
\begin{align}\label{eq: sum of binom upper bound}
    &\frac{{M \choose d} + {M \choose d-1} + \hdots + {M \choose 0}}{{M\choose k}}
    \leq \sum_{j=0}^{\infty} \left( \frac{d}{M-d+1} \right)^j = \frac{M-d+1}{M-2d+1}.
\end{align}
Therefore, we see that \(\sum_{i=0}^{d}{M\choose i} \leq {M\choose d} \cdot \frac{M-d+1}{M-2d+1}\).

When \(M \gg d\), we can upper bound \eqref{eq: number of cells} by \(C_1\binom{M-1}{d-1}\) which by a Sterling bound is bounded above by \(C_2(eM/d)^d\), where \(C_1, C_2 > 0\) are absolute constants. With so much redundancy a natural question to ask is how many faces\footnote{Here, we are referring to faces of maximal dimension, which are referred to as facets in other works.} delimit a cell in such a tessellation? Further, how large are the radii of the cells? The first question has seen a recent resurgence of interest \cite{hug2016random, hug2021another, godland2020random} and was initially addressed by Cover and Efron \cite{cover1967geometrical}. More specifically, it was shown in \cite[Eq. (4.1)]{cover1967geometrical} that under the assumption that the \(\varphi^{(i)}\) are in general position, the expected number of faces for a cell drawn uniformly at random from the tessellation is
\begin{align}\label{eq: expected number of faces}
    \frac{2 \cdot M \cdot \sum_{i=0}^{d-2} {M-2\choose i}}{\sum_{i=0}^{d-1} {M-1\choose i}}.
\end{align}
Using \eqref{eq: sum of binom upper bound} to upper bound \eqref{eq: expected number of faces} gives us
\begin{align}\label{eq: expected number of faces, simplified}
     &\frac{2 \cdot M \cdot \sum_{i=0}^{d-2} {M-2\choose i}}{\sum_{i=0}^{d-1} {M-1\choose i}} \leq \frac{2 \cdot M \cdot {M-2 \choose d-2} \cdot \frac{M-d+1}{M-2d+3}}{{M-1 \choose d-1}}\nonumber\\
    &\leq C_3 \cdot 2 \cdot M \cdot \frac{(M-2)!}{(d-2)!(M-d)!} \cdot \frac{(d-1)!(M-d)!}{(M-1)!} \nonumber\\
    &= C_3 \cdot 2 \cdot M \cdot \frac{d-1}{M-1} \leq C_4 d,
\end{align}
where, again, \(C_3, C_4 > 0\) are absolute constants. In other words, we can expect there to be no more than on the order of \(d\) hyperplanes to form a cell drawn uniformly from our tessellation when there are \(M \gg d\) hyperplanes. In fact, we can find a matching lower bound by a similar argument which proves that the average cell has a number of faces that scales like \(d\).

Unlike the setting of \cite{cover1967geometrical} we consider these two questions in the specific case when the \(\varphi^{(i)}\) are distributed according to a Gaussian distribution. We seek an upper bound (that holds with high probability on the draw of the \(\varphi^{(i)}\))  on the radius of cells formed by a collection of hyperplanes whose cardinality, up to constants and logarithmic terms, matches the bound \eqref{eq: expected number of faces, simplified}. Initially, we refer to a fixed cell by first fixing a point \(x \in S^{d-1}\), drawing the vectors \(\varphi^{(i)}\) independently of \(x\), and considering the cell containing $x$. We then pass to a uniform result where the result holds with high probability for all points \(x\) on the sphere.

\section{Background and connections} 
\noindent
The geometric questions listed above play an important role in the theory of quantization for finite frame expansions (see, e.g., \cite{goyal1995quantization, powell2013quantization}) and binary classification with linear separators \cite{vapnik1974theory}. These two contexts are dual to one another in the sense that for binary classifiers with linear separators the vector \(x\) is thought of as normal to a separating hyperplane which classifies the vectors \(\varphi^{(i)}\), into one of two classes, under the mapping \(\varphi^{(i)} \mapsto \sign(\langle x, \varphi^{(i)}\rangle)\). Notice that the mapping \(\varphi^{(i)} \mapsto \sign(\langle x, \varphi^{(i)}\rangle)\) is invariant under positive scalings of \(x\) so that we may assume without loss of generality \(\|x\|_2 = 1\). Frame quantization on the other hand views this problem in the setting as we have phrased it where the \(\varphi^{(i)}\) induce a spherical hyperplane tessellation and the goal is to recover an approximation of \(x\) from quantized measurements \(\sign(\langle x, \varphi^{(i)}\rangle)\). For example, such an approximation may take the form of a vector ${x^\sharp}$ from the same cell as $x$.

Our geometric intuition concerning the radii of the cells in the tessellation is guided by a result from quantization of finite frame expansions. In that context, the so-called measurement vectors \(\varphi^{(i)}\) are independently and randomly distributed, exhibit a correlation structure arising from structured random matrix ensembles, or arise from some deterministic construction. It was in this setting that Goyal et al. proved that if the \(\varphi^{(i)}\) form a frame (i.e., a spanning set for $\R^d$), not necessarily random, and \(x\) is taken to be random according to some distribution \(p\) over a set \(D\) then the root mean-squared error for recovering \(x\) using \(M\) quantized frame coefficients \((\sign(\langle x, \varphi^{(i)}\rangle))_{i=1}^{M}\) can decay \textit{no faster than} \(C \frac{1}{M}\) where \(C\) is a constant that depends on \(d, p, D\), but not $M$ \cite[Proposition~2]{goyal1995quantization}. The \(\varphi^{(i)}\) form a frame with probability \(1\) when they are i.i.d. Gaussian. Consequentially, this result essentially tells us that the largest cell in the induced tessellation has an inscribed ball of radius on the order of \(\frac{1}{M}\). The fact that the quantization error can only decay linearly, rather than exponentially with $M$, has motivated the design of more sophisticated quantization schemes, such as $\Sigma \Delta$ or noise shaping quantizers (see, e.g, \cite{powell2013quantization, chou2015noise}), and non-adaptive universal one-bit quantizers \cite{boufounos2011universal}. Herein, we will demonstrate that for any \(x \in S^{d-1}\) there exists a sub-collection of \(C_5 d\log(d)\log(M)\) hyperplanes such that the radius of the cell containing \(x\) induced by this sub-collection is at most on the order of  \(d\log(d)\log(M)/M\). In other words, relatively few vectors $\varphi^{(i)} $ identify the cell.

Binary classification with linear separators is often known under the moniker of the support vector machine. The support vector machine was first proposed in \cite{vapnik1974theory}. Various works have analyzed this related problem. Importantly, the results that we are aware of differ slightly from the geometric context that we have posed. Namely, most support vector machine results are concerned with the scenario where the spherical hyperplane tessellation is not fixed beforehand. In this dual setting, results are typically formulated in terms of how many point samples \(\varphi^{(i)}\) with labels \(\sign(\langle \varphi^{(i)}, x\rangle)\) are required to learn a separating hyperplane with normal vector \(x^\sharp\) so that the probability of a new point \(\varphi_{M+1}\) being incorrectly classified according to \(\sign(\langle \varphi_{M+1}, x^\sharp \rangle)\) is less than \(\eps\). In the primal setting, a separating hyperplane \(x^\sharp\) is any point in the same cell as \(x\) induced by the spherical tessellation from \(\varphi^{(1)}, \hdots, \varphi^{(M)}\). Learning an \(\eps\)-accurate classifier amounts to a choice of \(x^\sharp\) so that with probability less than \(\eps\) a newly drawn hyperplane with normal vector \(\varphi^{(M+1)}\) excludes \(x^\sharp\) from the new cell with \(x\) in it. For further reading, one may look at \cite{long1995sample, angluin1988queries, campbell2000query, tong2001support, balcan2007margin, dasgupta2008general} for example.

\section{Notation}
\label{sec:problem}
\noindent
Henceforth, \(\langle \cdot, \cdot \rangle\) denotes the standard inner product over \(\R^d\). \(S^{d-1}\) and \(B_2^d(x, \eps)\) denote the unit sphere and the ball of radius \(\eps\) centered at \(x\) in \(\R^d\), respectively. The vector \(e_i \in S^{d-1}\) denotes the \(i^{th}\) standard basis vector in \(\R^d\) and the matrix $I_d$ denotes the $d \times d$ identity matrix. For any vector \(z \in \R^d\), define the vector \(z_{[-1]} \in \R^{d-1}\) to be the projection of \(z\) which removes the first coordinate and  let \(z_j \in \R\) denote the $j^{th}$ entry of $z$. For any integer \(n \in \N\), we let \([n] = \{1, 2, \hdots, n\}\). The relation \(a \gtrsim b\) denotes that there exists some universal constant \(C > 0\) so that \(a \geq Cb\). Similarly, the relation \(a \gg b\) denotes that there exists some universal constant \(C > 1\) so that \(a \geq Cb\). In order to avoid cumbersome notation, we reserve the symbols \(C, C', c, c'\) and so forth for universal constants. For a random variable $X$, we define the Orlicz norm (see \cite[Section~2.7.1]{vershynin2018high}) as \begin{align*}
         \|X\|_{\psi_2} = \inf \left\{ \omega > 0 : \E e^{X^{2}/\omega^2} \leq 2 \right\},
    \end{align*} 
while for a random vector $X\in\R^d$, we define it as 
\begin{align*}
         \|X\|_{\psi_2} = \sup_{x\in S^{d-1}} \|\langle X, x\rangle \|_{\psi_2}.
    \end{align*}

With this notation, we define and approach our problem as follows. Let \(x \in S^{d-1}\) be some fixed unit vector in \(\R^d\), and suppose we draw \(M \gg d\) i.i.d. vectors \(g^{(i)} \sim \Nc(0, d^{-1} I)\) at random and are given the sign pattern
\begin{align*}
    y_i = \sign(\langle g^{(i)}, x \rangle), \quad i \in [M].
\end{align*}
The discussion in the introductory section tells us that the \(M\) measurements are highly redundant and that one should be able to infer which cell \(x\) is in from much fewer measurements. Given the derivation in \eqref{eq: expected number of faces, simplified} we might expect the number of hyperplanes forming the cell containing \(x\) to scale like \(d\). In light of this, we consider sub-selecting a set $S \subset [M]$ of size \(m \ll M\) from the i.i.d. vectors $g^{(i)}$ so that with high probability
\begin{align*}
    \sign(\langle g^{(i)}, x \rangle) = \sign(\langle g^{(i)}, y \rangle)~\forall~  i\in S\implies \|x - y\|_2 \text{ is small.}
\end{align*}
More specifically, we consider sets of the form
\begin{equation}\label{eq:Stau}
    T_{\tau, x} := \{i \in [M] : |\langle g^{(i)}, x \rangle| < \tau\}.
\end{equation}

\noindent
Our goal is to choose \(\tau\) so that
\(T_{\tau,x}\) is large enough to define a spherical polytope containing $x$ with a radius on the order of the smallest cell containing \(x\). That is, we would like to choose \(\tau\) so that \(\{y \in S^{d-1} : \sign(\langle \varphi^{(i)}, x \rangle) = \sign(\langle \varphi^{(i)}, y \rangle) \,\, \text{for all } i \in [M]\} \subset \{y \in S^{d-1} : \sign (\langle \varphi^{(i)}, x \rangle) = \sign(\langle \varphi^{(i)}, y \rangle) \,\, \text{for all } i \in T_{\tau, x}\}\). On the other hand, we also wish \(T_{\tau,x}\) to have cardinality roughly on the order of \(d\). Our main results show that, indeed, this is possible with high probability.

\section{Results}
\noindent
The main result of \cite[Proposition~2]{goyal1995quantization} tells us that using \(M\) hyperplanes there exists a point \(x \in S^{d-1}\) and a point \(z\) in the same cell as \(x\) which satisfy \(\|x - z\|_2 \gtrsim \frac{d}{M}\). Theorem \ref{thm: main} shows that if we were to have access to the sign patterns from the set \(T_{\tau, x}\), then it is possible to find a matching upper bound up to an additional factor of \(\log(d)\log(M)\) with only \(C d\log(d)\log(M)\) hyperplanes. 

\label{sec:results}
\begin{theorem}\label{thm: main}  Fix $x \in S^{d-1}$  and draw i.i.d. vectors $g^{(i)} \sim \mathcal{N}(0, d^{-1}I_{d})$ for $i \in [M]$ with \(\frac{M}{\log(M)} \gtrsim d \log(d)\). 
Select once and for all a collection of \(C_1 d\log(d)\) indices independent of \(x\) without replacement from \([M]\) and put them in a set called \(V\). Choose $$\tau = C_2 \frac{\sqrt{d} \log(d) \log(M)}{M},$$ and define the sets \(\W := \{i \in [M] : -\tau < \langle g^{(i)}, x \rangle < 0\}\), \(S_{\tau,x } := \W \cup V\). Then, with probability at least $1 - 3\exp (-c d \log (d)) - \exp(-c' d\log(M))$ the following is true: \begin{align*}
    C_3 d\log(d)\log(M) \leq |S_{\tau,x }| \leq C_4 d\log(d)\log(M),
\end{align*}
and if, for $y \in S^{d-1}$, \( \sign(\langle g^{(i)}, x \rangle) = \sign(\langle g^{(i)}, y \rangle) \text{ for all } i\in S_{\tau,x}\) then 
\begin{align}\label{eq: cell radius bound, fixed x}
    \|x - y\|_2 \leq \frac{C_5 d\log(d)\log(M)}{\sqrt{M^2 + d^2\log^2(d)\log^2(M)}}.
\end{align}
\end{theorem}
 \noindent
 We can understand Theorem \ref{thm: main} using some simple geometric intuition. The vectors \(g^{(i)}\) whose inner products with \(x\) are the smallest in absolute value determine the normal hyperplanes which are closest to \(x\). So, these hyperplanes are the most informative with respect to determining the faces of the cell in the full tessellation induced by all of the vectors \(g^{(1)}, \hdots, g^{(M)}\). However, because we only have finitely many \(g^{(i)}\) we need to choose \(\tau\) carefully so that the set \(\W = \{i \in [M] : -\tau < \langle g^{(i)}, x \rangle < 0\}\) is non-empty. 
 
 Our proof of Theorem \ref{thm: main} relies on reducing the bound \eqref{eq: cell radius bound, fixed x} to a lower bound on the minimal singular value of a particular random matrix. To see how, we provide the following quick sketch. By rotation invariance of the \(g^{(i)}\), let's assume \(x = e_1\). Notice we have \(\|x - y\|_2^2 = 2 - 2 y_1\) for any \(y \in S^{d-1}\). So, to control \(\max_y\|x-y\|_2^2\) where \(y\) is such that \( \sign(\langle g^{(i)}, x \rangle) = \sign(\langle g^{(i)}, y \rangle) \text{ for all } i\in S_{\tau,x}\), we just need a uniform lower bound on \(y_1\) for such \(y\). This lower bound will come from the feasibility condition
\begin{align}\label{eq: feasibility}
    \langle g^{(i)}, x\rangle \langle g^{(i)}, y\rangle = g^{(i)}_1  \langle g^{(i)}, y\rangle > 0  \,\, \text{ for all } i \in S_{\tau, x} \nonumber\\
    \iff(g^{(i)}_1)^2 y_1 > - g^{(i)}_1 \langle g^{(i)}_{[-1]}, y_{[-1]} \rangle, \,\, \text{ for all } i \in S_{\tau, x}.
\end{align}
 Now, assuming \(g_1^{(i)} \neq 0\), \textit{if we could preserve \eqref{eq: feasibility} by squaring both sides -- for example if both sides were both positive --} we would end up with an inequality of the form
 \begin{align}\label{eq: eric dreams of squaring}
     g^{(i) 2}_1 y_1^2 > \|y_{[-1]}\|_2^2 \left\langle g^{(i)}_{[-1]},  \frac{y_{[-1]}}{\|y_{[-1]}\|_2} \right\rangle^2,  \,\, \text{ for all } i \in S_{\tau, x}.
 \end{align}
 Using the fact that \(\|y_{[-1]}\|_2^2 = 1 - y_1^2\) and \(|g^{(i)}_1| < \tau\) on the subset of vectors \(\W \subset S_{\tau, x}\) as defined in Theorem \ref{thm: main}, we could collect like terms to get
 \begin{align*}
    y_1^2 \geq 1 - \frac{\tau^2}{\left\langle g^{(i)}_{[-1]}, \frac{y_{[-1]}}{\|y_{[-1]}\|_2} \right\rangle^2 + \tau^2}, \,\, \text{ for all } i \in \W \subset S_{\tau, x}.
 \end{align*}
 From here, all we would need to do is find a uniform lower bound for \(\left\langle g^{(i)}_{[-1]}, \frac{y_{[-1]}}{\|y_{[-1]}\|_2} \right\rangle^2\) which is where random matrix theory would come into play. The only way this strategy works is if we could  show that the inequalities \eqref{eq: feasibility} and \eqref{eq: eric dreams of squaring} were essentially equivalent. Of course, if both sides of \eqref{eq: feasibility} were positive we would be in great shape. The terms we have to worry about are \(y_1\) and \(\langle g^{(i)}_{[-1]}, y_{[-1]}\rangle\). Our goal is to show that we can throw away a constant proportion of the vectors in \(S_{\tau, x}\) where the two conditions \(y_1 > 0\) and \(\langle g^{(i)}_{[-1]}, y_{[-1]}\rangle > 0\) do not hold and still get the result of Theorem \ref{thm: main}. Lemma \ref{lem: feasible y, in the same half space UNIFORM} handles the condition \(y_1 > 0\) and Corollary \ref{cor: tail inner products positive, fixed y} handles the  condition \(\langle g^{(i)}_{[-1]}, y_{[-1]}\rangle > 0\). Besides these two results, proving Theorem \ref{thm: main} requires two main ingredients. First, Lemma \ref{lem: S_tau}
 characterizes how \(\tau\) should be chosen to ensure that, with high probability, \(|S_{\tau,x}|\) is about \( d\log(d)\log(M)\). This, along with Lemma \ref{lem: smallest singular value Phi_y, uniform} finish the heavy lifting needed to lower bound \(\left\langle g^{(i)}_{[-1]}, \frac{y_{[-1]}}{\|y_{[-1]}\|_2} \right\rangle^2\). If not defined explicitly in the statements below, we note that variables are as in Theorem \ref{thm: main}.
 
\begin{lemma}\label{lem: S_tau} If $\tau = C_2 \frac{\sqrt{d}\log(d) \log(M)}{M}$, then with probability at least $1 - \exp(-c d\log(d) \log(M))$ we have \(C_3d\log(d)\log(M) \leq |S_{\tau,x}| \leq C_4d\log(d)\log(M)\).
\end{lemma}
\begin{proof}
As in the statement of Theorem \ref{thm: main} let \(\W := \{i \in [M] : -\tau < \langle g^{(i)}, x \rangle < 0\}\). Note that \(\W \subset S_{\tau, x}\) and 
\begin{align}\label{eq: S_t}
    |\W| = \sum_{i=1}^{M} \mathbb{1}\left( |\langle g^{(i)}, x\rangle| < \tau \right) \1\left( \langle g^{(i)}, x\rangle < 0 \right).
\end{align}
 We remark that the events defined in the indicator functions are independent. Further, each \(\1\left( |\langle g^{(i)}, x\rangle| < \tau \right)\) is a Bernoulli random variable with mean \(\P\left( |\langle g^{(i)}, x\rangle| < \tau\right)\). By our choice of normalization, \(\langle g^{(i)}, x\rangle~\sim~\Nc(0,d^{-1})\), so we have \(\P\left( |\langle g^{(i)}, x\rangle| < \tau\right) = \erf\left(\frac{\tau\sqrt{d}}{\sqrt{2}}\right)\), where \(\erf\) denotes the Gauss error function
\begin{align*}
    \erf(t) = \frac{1}{\sqrt{\pi}}\int_{-t}^{t} e^{-x^2}\,\,dx.
\end{align*} 
Using a multiplicative Chernoff bound \cite{vershynin2018high}
for any \(s \in (0,1]\)
\begin{align*}
    \P&\left( \left| |\W| - \frac{M}{2}\erf\left(\frac{\tau\sqrt{d}}{\sqrt{2}}\right) \right| \geq s \frac{M}{2}\erf\left(\frac{\tau\sqrt{d}}{\sqrt{2}}\right) \right) \\
    &\leq 2\exp\left(-c s^2 M\erf\left(\frac{\tau\sqrt{d}}{\sqrt{2}}\right) \right) \leq 2\exp\left(-c s^2 M \tau \sqrt{d} e^{-\tau^2 d/2}\right).
\end{align*}
More specifically, this tells us that with high probability \(|\W| \geq c' M\erf\left(\frac{\tau\sqrt{d}}{\sqrt{2}}\right) \geq c^{''}M\tau\sqrt{d} \exp(-\tau^2 d/2) \geq c^{'''}d\log(d) \log(M)\) using a lower Riemann sum estimate to lower bound \(\erf\) in the penultimate inequality. Notice that the term $e^{-\tau^2 d/2}$ does not pose any problems, since \(\tau^2 d/2 = C_2 d\log(d)\log(M)/2M\) and we assume that \(M \gtrsim d\log(d)\log(M)\). Analogously, we also have using an upper Riemann sum estimate \(|\W| \leq C M\erf\left(\frac{\tau\sqrt{d}}{\sqrt{2}}\right) \leq C'M\tau\sqrt{d} = C''d\log(d)\log(M)\). Since \(S_{\tau, x} = \W \cup V \) and \(|V| = C_1 d\log(d)\), the claim then follows.
\end{proof}

\begin{lemma}\label{lem: feasible y, in the same half space UNIFORM}
    Let \(A \in \R^{C_1d\log(d) \times d}\) be a matrix whose rows are populated with the vectors in \(V\), which is defined in the statement of Theorem \ref{thm: main}. Then for any \(x, y \in S^{d-1}\), the following is true with probability exceeding \(1-e^{-cd\log(d)}\): \(\sign(Ax) = \sign(Ay)\) implies \(\|x-y\|_2 < \sqrt{2}\), and consequently \(\langle x, y \rangle > 0\).
\end{lemma}
\begin{proof}
    This is a direct application of Theorem 2 in \cite{jacques2013robust} considering the measurements formed using the matrix \(A\).
\end{proof}
\noindent
Recall that to achieve our goal to transition from \eqref{eq: feasibility} to \eqref{eq: eric dreams of squaring}, it would suffice to have \(y_1 > 0\) and \(\langle g^{(i)}_{[-1]}, y_{[-1]}\rangle > 0\) for all \(i \in S_{\tau, x}\) and all \(y\) in the same cell as \(x\). By nature of the latter condition being a ``for all \(y\)'' statement, we will follow a relatively standard argument relying on a probabilistic union bound coupled with a continuity argument. To make the continuity argument work, we will use  the slightly stronger condition \(\langle g^{(i)}_{[-1]}, y_{[-1]}\rangle > \eta > 0\) for \(\eta\) to be defined below.
\begin{corollary}\label{cor: tail inner products positive, fixed y}
     Let \(x=e_1\), the first standard basis vector,  fix \(y \in S^{d-1}\), and suppose that \(\|y_{[-1]}\|_2^2 \geq \frac{1}{M^2}\). Define \(\W \subset S_{\tau, x}\) as in Theorem \ref{thm: main}, and set \(\eta =  C_5 \frac{\log(M)}{M^2}\). Consider the subset \(\SSS := \{i \in \W : \langle g^{(i)}_{[-1]}, y_{[-1]} \rangle > \eta\} \subset S_{\tau, x}\). Then with probability at least \(1-e^{-c d\log(d)\log(M)}\), we have \(\widetilde{C}_3d\log(d)\log(M) \leq |\SSS| \leq \widetilde{C}_4d\log(d)\log(M)\).
\end{corollary}
\begin{proof}
    As before, we have
    \begin{align*}
        |\SSS| := \sum_{i\in[M]} \1(-\tau \leq \langle g^{(i)}, x \rangle \leq 0) \1(\langle g^{(i)}_{[-1]}, y_{[-1]} \rangle >\eta).
    \end{align*}
    Notice that \(\P \left(\langle g^{(i)}_{[-1]}, y_{[-1]} \rangle > \eta \right) \geq \P\left( \gamma \geq  \eta M\sqrt{d} \right)\), where \(\gamma \sim \Nc(0,1)\) and that, as in ~\eqref{eq: S_t}, the events defined in the indicator functions are independent.
    Using a Riemann sum to approximate this probability, we have that this aforementioned probability is lower bounded by \(\frac{1}{2} - c' \eta M\sqrt{d} = \frac{1}{2} - C' \frac{\sqrt{d}\log(M)}{M} = \tilde{c} > 0\). So, we find \(\E|\SSS| =  \tilde{c}\E |\W|\), and using a Chernoff bound as in Lemma \ref{lem: S_tau} the claim follows.
\end{proof}

The proof of the Lemma~\ref{lem: smallest singular value Phi_y, uniform} depends on the following intermediate computation, whose statement and proof are provided in Lemma~\ref{lem: minimum singular value},  Section~\ref{sec:app}.

\begin{lemma}\label{lem: smallest singular value Phi_y, uniform}
 Fix \(x = e_1\). Then the following is true for all \(y \in S^{d-1}\) with \(\|y_{[-1]}\|_2^2 \geq \frac{1}{M^2}\): define \(S_{\tau, x}^{y} := \{i \in \W : \langle g^{(i)}_{[-1]}, y_{[-1]}\rangle > 0\}\), and define the matrix \(G_{y} \in \R^{|S_{\tau, x}^{y}| \times (d-1)}\) by populating its rows with \(g^{(i)}_{[-1]}\), where \(i \in S_{\tau, x}^{y}\). Then with probability \(1 - e^{-c'd \log(d) \log(M)} - 2e^{-c''d\log(M)}\), we have
    \begin{align*}
        \min_{z\in S^{d-2}} \|G_{y} z\|_2^2 \gtrsim \log(d)\log(M).
    \end{align*}
\end{lemma}

\begin{proof}
    To start, fix \(y\)  and set \(\eta =  C_5 \frac{\log(M)}{M^2}\). By rotational invariance, we may assume \(y_{[-1]}\) is parallel to \(e_1 \in \R^{d-1}\). Define the matrix \(\tilde{G}_{y} \in \R^{|\SSS| \times (d-1)}\) by populating its rows with \(g^{(i)}_{[-1]}\), where \(i \in \SSS\) and \(\SSS\) is defined in Corollary \ref{cor: tail inner products positive, fixed y}. By that very same corollary, \(\widetilde{C}_3d\log(d)\log(M) \leq |\SSS| \leq \widetilde{C}_4d\log(d)\log(M)\) with probability at least \(1 - e^{-c_1 d\log(d)\log(M)}\). Since any orthogonal projection of a standard Gaussian random vector is itself a Gaussian random vector, the rows of \(\tilde{G}_{y}\) are independent and identically distributed. Within each row, the first entry is distributed according to the conditional law \(g | g > \frac{\eta}{\|y_{[-1]}\|_2}\) and the remaining entries are i.i.d. Using Lemma \ref{lem: minimum singular value} with \(t = \sqrt{c_2 d\log(M)}\), we have with probability at least \(1 - e^{-c_2 d\log(M)}\),
    \begin{align*}
        \Big\| \frac{1}{|\SSS|} \tilde{G}_{y}^T\tilde{G}_{y} - \Sigma \Big\|_{op} \leq C K^2 \left(\sqrt{\frac{d}{|\SSS|}} + \sqrt{\frac{c_2d\log(M)}{|\SSS|}}\right),
    \end{align*}
    where \(K = \max \{ \|g^{(i)}_{[-1]}\|_{\psi_2} : i \in \SSS\}\) is the maximal Orlicz \(2\)-norm of the rows and \(\Sigma = \E g_{[-1]}g^{T}_{[-1]} \in \mathbb{R}^{(d-1) \times (d-1)}\). We will control both of these quantities shortly. For now though, notice that this immediately gives us the following lower bound on the minimal singular value of \(\tilde{G}_{y}\), which holds with high probability, by the following algebra:
    \begin{align*}
        &\frac{1}{|\SSS|} \inf_{\|z\|_2 = 1}\| \tilde{G}_{y} z \|_2^2
        = \frac{1}{|\SSS|} \inf_{\|z\|_2 = 1}\| \tilde{G}_{y}^T \tilde{G}_{y} z \|_2\\
        &=  \inf_{\|z\|_2 = 1}\left\|\left(\frac{1}{|\SSS|}\tilde{G}_{y}^T \tilde{G}_{y} - \Sigma + \Sigma \right) z \right\|_2\\
        &\geq \inf_{\|z\|_2 = 1}\|\Sigma z \|_2 - \sup_{\|z\|_2 = 1}\left\|\left(\frac{1}{|\SSS|}\tilde{G}_{y}^T \tilde{G}_{y} - \Sigma\right) z \right\|_2\\
        &\geq \inf_{\|z\|_2 = 1}\|\Sigma z \|_2 - C K^2 \left(\sqrt{\frac{d}{|\SSS|}} + \sqrt{\frac{c_2d\log(M)}{|\SSS|}}\right).
    \end{align*}
    Now, we quickly calculate what the minimal singular value of \(\Sigma \) is. In the following calculations, let \(\gamma~\sim~\Nc(0, 1)\):
    \begin{align*}
        \Sigma &= \E g_{[-1]}g^T_{[-1]} = 
        \begin{bmatrix}
        \alpha & & &\\
        & d^{-1} & &\\
        & & \ddots &\\
        & & & d^{-1}
        \end{bmatrix},\\
        \alpha &= \frac{1}{\P\left( \sqrt{d^{-1}} \gamma \geq \frac{\eta}{\|y_{[-1]}\|_2} \right)} \frac{1}{\sqrt{2 \pi d^{-1}}} \int_{\frac{\eta}{\|y_{[-1]}\|_2}}^{\infty} x^2 e^{-\frac{x^2}{2d^{-1}}}\,\, dx.
    \end{align*}
    We can use integration by parts to expand this integral as
    \begin{align}\label{eq: integration by parts}
        \int_{\frac{\eta}{\|y_{[-1]}\|_2}}^{\infty} x^2 e^{-\frac{x^2}{2d^{-1}}}\,\, dx = \frac{\eta e^{\frac{-\eta}{2\|y_{-1}\|_2 d^{-1}}}}{\|y_{-1}\|_2 d^{-1}} + d^{-1} \int_{\frac{\eta}{\|y_{-1}\|_2}}^{\infty} e^{- \frac{x^2}{2d^{-1}}} \,\, dx.
    \end{align}
    Notice that when we multiply the second term in \eqref{eq: integration by parts} by \(\frac{1}{\P\left( \sqrt{d^{-1}} \gamma \geq \frac{\eta}{\|y_{[-1]}\|_2} \right)} \frac{1}{\sqrt{2 \pi d^{-1}}}\) it simplifies to
    \begin{align*}
        \frac{1}{\P\left( \sqrt{d^{-1}} \gamma \geq \frac{\eta}{\|y_{[-1]}\|_2} \right)} \frac{1}{\sqrt{2 \pi d^{-1}}}  d^{-1} \int_{\frac{\eta}{\|y_{-1}\|_2}}^{\infty} e^{- \frac{x^2}{2d^{-1}}} \,\, dx\\
        = d^{-1} \frac{1}{\P\left( \sqrt{d^{-1}} \gamma \geq \frac{\eta}{\|y_{[-1]}\|_2} \right)} \P\left( \sqrt{d^{-1}} \gamma \geq \frac{\eta}{\|y_{[-1]}\|_2} \right) = d^{-1}.
    \end{align*}
    Therefore, we have
    \begin{align*}
        \alpha = d^{-1} + \frac{\sqrt{d^{-1}}\frac{\eta}{\|y_{[-1]}\|_2} e^{-\frac{1}{2d^{-1}}\left(\frac{\eta}{\|y_{[-1]}\|_2}\right)^2}}{\sqrt{2\pi}\P\left( \sqrt{d^{-1}} \gamma \geq \frac{\eta}{\|y_{[-1]}\|_2} \right)} > d^{-1}.
    \end{align*}
    As for \(K\), by independence of the entries of \(g^{(i)}_{[-1]}\) we have using Lemma 2.6.8 and Lemma 3.4.2 from \cite{vershynin2018high}
    \begin{align*}
        \|g^{(i)}_{[-1]}\|_{\psi_2} \leq C \max_{j \in \{2, 3, \hdots, d\}} \|g^{(i)}_{j}\|_{\psi_2},
    \end{align*}
    so it suffices to upper-bound the Orlicz \(2\)-norm of the first entry \(g^{(i)}_{[-1],1}\). By definition, we have
    \begin{align*}
         \|g^{(i)}_{[-1],1}\|_{\psi_2} = \inf \left\{ \omega > 0 : \E e^{g^{(i)2}_{[-1],1}/\omega^2} \leq 2 \right\}.
    \end{align*}
    Note that when \(\omega = 2\sqrt{d^{-1}}\), we have
    \begin{align*}
        &\E e^{g^{(i)2}_{[-1],1}/\omega^2} =  \frac{1}{\P\left( \sqrt{d^{-1}} \gamma \geq \frac{\eta}{\|y_{[-1]}\|_2} \right)} \frac{1}{\sqrt{2 \pi d^{-1}}} \int_{\frac{\eta}{\|y_{[-1]}\|_2}}^{\infty} e^{\frac{x^2}{4d^{-1}}}e^{\frac{-x^2}{2d^{-1}}}\,\, dx \\
        &= \frac{\sqrt{2}}{\P\left( \sqrt{d^{-1}} \gamma \geq \frac{\eta}{\|y_{[-1]}\|_2} \right)} \frac{1}{\sqrt{2 \pi (\sqrt{2}d^{-1/2})^2}} \int_{\frac{\eta}{\|y_{[-1]}\|_2}}^{\infty} e^{\frac{-x^2}{2 (\sqrt{2}d^{-1/2})^2}}\,\, dx\\
        &= \frac{\sqrt{2}\P\left( \sqrt{2 d^{-1}} \gamma \geq \frac{\eta}{\|y_{[-1]}\|_2} \right)}{\P\left( \sqrt{d^{-1}} \gamma \geq \frac{\eta}{\|y_{[-1]}\|_2} \right)} = \frac{\sqrt{2}\P\left( \gamma \geq \frac{\eta\sqrt{d}}{\|y_{[-1]}\|_2\sqrt{2}} \right)}{\P\left(\gamma \geq \frac{\eta\sqrt{d}}{\|y_{[-1]}\|_2} \right)}.
    \end{align*}
    One may numerically check that the above quantity is less than \(2\) provided \(\frac{\eta\sqrt{d}}{\|y_{[-1]}\|_2} \leq \frac{C \sqrt{d}\log(M)}{M} < \frac{1}{2}\). Since \(M \gtrsim d\log(d)\log(M)\), the latter inequality will hold provided the constant in the lower bound of \(M\) is large enough. Hence, \(\|g^{(i)}_{[-1],1}\|_{\psi_2} \lesssim \sqrt{d^{-1}}\). Consequentially, we have shown with probability at least \(1 - e^{-c_2 d\log(M)} - e^{-c_1d\log(d)\log(M)}\) that
    \begin{align}\label{eq: lower bound sing val fixed y}
    \frac{1}{|\SSS|} \inf_{\|z\|_2 = 1}\| \tilde{G}_{y} z \|_2^2 &\geq d^{-1} - C d^{-1} \left(\sqrt{\frac{d}{|\SSS|}} + \sqrt{\frac{c_2d\log(M)}{|\SSS|}}\right)\nonumber\\
    &\geq d^{-1}\left( 1 - C' \left(\sqrt{\frac{1}{\log(d)\log(M)}} + \sqrt{\frac{1}{\log(d)}}\right)\right)\nonumber\\
    \implies \inf_{\|z\|_2 = 1}\| \tilde{G}_{y} z \|_2^2 &\gtrsim d^{-1}|\SSS| \gtrsim \log(d)\log(M).
    \end{align}
    To get the uniform result over all \(y\), condition on the event \(\Ec\) where \(\max_{i \in [M]} \|g^{(i)}\|_2 \leq C' \sqrt{\log(M)}\), which occurs with probability at least \(1 - \exp(-c_3 d\log(M))\) \cite{vershynin2018high}. Let \(\eps = \frac{1}{M^2}\) and \(\Nc\) be an \(\eps\)-net of \(S^{d-1}\). We know that the cardinality of this net is bounded by \(|\Nc| \leq \left(3M^2\right)^d \leq e^{c_4 d\log(M)}\) \cite{vershynin2018high}. Union bounding over this net, we have that the event \(\Ec\), the result of Corollary \ref{cor: tail inner products positive, fixed y} and \eqref{eq: lower bound sing val fixed y} hold for \(\tilde{G}_{\xi}\) for all \(\xi \in \Nc\) simultaneously with probability at least \[1 - e^{-c_1d \log(d) \log(M)} - e^{-c_5d\log(M)} - e^{-c_3 d\log(M)}.\] Here, \(c_5 = c_2 - c_4 > 0\) provided we choose \(c_2 > c_4\) in our choice of parameter \(t\) in Lemma \ref{lem: minimum singular value}. On this event, for any \(y \in S^{d-1}\setminus \Nc\) we have some \(\xi \in \Nc\) with \(\|y - \xi\|_2 \leq \eps\) and therefore for any \(i \in \tilde{S}^{\xi}_{\tau, x}\), we have
    \begin{align*}
        \langle g^{(i)}_{[-1]}, y_{[-1]} &\rangle \geq \langle g^{(i)}_{[-1]}, \xi_{[-1]} \rangle - \|g^{(i)}_{[-1]}\|_2 \eps \\ 
        &\geq \eta - C' \frac{\sqrt{\log(M)}}{M^2} > 0,
    \end{align*}
    where the above inequality holds, for example, provided \(C_5\) in \(\eta = C_5 \log(M)/M^2\) satisfies \(C_5 > C'\). So, in other words, we have \(\tilde{S}^{\xi}_{\tau, x} \subset S^{y}_{\tau, x}\). This immediately implies that \(\inf_{\|z\|_2 = 1} \| G_{y} z\|_2^2 \geq \inf_{\|z\|_2 = 1} \|\tilde{G}_{\xi} z\|_2^2 \gtrsim  d^{-1}|\tilde{S}^{\xi}_{\tau, x}| \gtrsim  \log(d)\log(M)\).
\end{proof}

\proof[Proof of Theorem \ref{thm: main}]
\noindent
    We are now ready to make our heuristic proof sketch rigorous. We focus on bounding \(\max \|x-y\|_2\) where the maximum is taken over all \(y \in S^{d-1}\) which satisfy \(\sign(\langle g^{(i)}, x\rangle) = \sign(\langle g^{(i)}, y\rangle)\) for all \(i \in S_{\tau, x}\). By rotational invariance, we may assume without loss of generality that \(x=e_1\). Note that if \(\|y_{[-1]}\|_2^2 \leq M^{-2}\) then our result automatically holds, so we focus only on those \(y\) for which said inequality is violated. Since \(\max \|x-y\|_2^2 = 2 - 2\min{y_1}\), we turn to finding a lower bound on \(\min{y_1}\). By feasibility, we know for any \(i \in S_{\tau, x}^y\)
\begin{align*}
    \langle g^{(i)}, x\rangle \langle g^{(i)}, y\rangle = g^{(i)}_1  \langle g^{(i)}, y\rangle > 0 \\
    \iff(g^{(i)}_1)^2 y_1 > - g^{(i)}_1 \langle g^{(i)}_{[-1]}, y_{[-1]}\rangle,
\end{align*}
where $S_{\tau ,x}^y$ is as defined in Lemma~\ref{lem: smallest singular value Phi_y, uniform}. By construction, both sides of this inequality are positive, as $g_1^{(i)}$ is negative and \(y_1 > 0\) with probability at least \(1-\exp(-c''''d\log(d))\) for all \(y\) in the same cell as \(x\) by Lemma \ref{lem: feasible y, in the same half space UNIFORM}. Squaring both sides of the inequality, using \(|g^{(i)}_1| < \tau\) for \(i \in S^y_{\tau, x} \subset \W\), and \(\|y_{[-1]}\|_2^2 = (1-y_1^2)\) means the set of feasible \(y's\) necessarily satisfy
\begin{align*}
    y_1^2 &\geq  1 - \frac{\tau^2}{\left\langle g^{(i)}_{[-1]}, \frac{y_{[-1]}}{\|y_{[-1]}\|_2} \right\rangle^2 + \tau^2}, \,\, \text{ for all } i \in S_{\tau, x}^y.
\end{align*}
In other words, to control \(\max \|x-y\|_2\) for \(y\) in the same cell as \(x\) induced by \(S_{\tau, x}\), it is sufficient to find a lower bound on the quantity
\begin{align*}
    \min_{z \in S^{d-2}} \max_{i\in S_{\tau, x}^y } \left\langle g^{(i)}_{[-1]}, z \right\rangle^2.
\end{align*}
Letting $G_y$ be as defined in Lemma~\ref{lem: smallest singular value Phi_y, uniform}, notice that we may lower bound this quantity by
\begin{align*}
    \min_{z \in S^{d-2}} \max_{i\in S_{\tau, x}^y} \left\langle g^{(i)}_{[-1]}, z \right\rangle^2 &\geq \min_{z \in S^{d-2}} |S_{\tau,x}^y|^{-1} \sum_{i \in S_{\tau, x}^y}  \left\langle g^{(i)}_{[-1]}, z \right\rangle ^2 \\
    &= |S_{\tau, x}^y|^{-1} \min_{z \in S^{d-2}} \|G_y z\|_2^2.
\end{align*}
By Lemma \ref{lem: smallest singular value Phi_y, uniform}, and Corollary \ref{cor: tail inner products positive, fixed y} we know with probability at least \(1 - e^{-c'd \log(d) \log(M)} - 2e^{-c''d\log(M)} \) \(\min_{z\in S^{d-2}} \|G_{y}z\|_2^2 \gtrsim \log(d)\log(M)\) and $|S_{\tau, x}^y|\lesssim d\log(d)\log(M)$ for all \(y\). Therefore, 
\begin{align*}
    \min_{z \in S^{d-2}} \max_{i\in S_{\tau, x}^y} \left\langle g^{(i)}_{[-1]}, z \right\rangle^2 \gtrsim d^{-1}.
\end{align*}
Combining our results from this section, we have with high probability
\begin{align*}
    y_1^2 &\geq 1 - \frac{\tau^2}{Cd^{-1} + \tau^2}.
\end{align*}
\noindent
With our choice of \(\tau = C_2 \frac{\sqrt{d}\log(d)\log(M)}{M}\), this means
\begin{align*}
    y_1^2 \geq 1 - C \frac{d\log^2(d)\log^2(M)}{\frac{M^2}{d} + d\log^2(d)\log^2(M)}.
\end{align*}
This concludes the proof of Theorem \ref{thm: main}.

\section{Uniform Results}
\noindent
As a result of working so hard to get uniform control on \(\langle g^{(i)}_{[-1]}, y_{[-1]} \rangle\) for all \(y\) in the same cell as \(x\) using a union bound argument, we can get a uniform result for all \(x \in S^{d-1}\) nearly for free. The consequence of uniformly bounding the radius of the cell around \(x\) induced by the vectors selected according to \(S_{\tau, x}\) gives an upper bound on the radii for all cells in the tessellation induced by the entire collection of vectors \(g^{(1)}, \hdots, g^{(M)}\). This follows from the inclusion \(\{y \in S^{d-1} : \sign(\langle g^{(i)}, y \rangle) = \sign(\langle g^{(i)}, x \rangle), \,\, \text{ for all } i \in [M]\} \subset \{y \in S^{d-1} : \sign(\langle g^{(i)}, y \rangle) = \sign(\langle g^{(i)}, x \rangle), \,\, \text{ for all } i \in S_{\tau, x}\}\).
\begin{corollary}\label{cor: uniform result}
Let $g^{(i)} \in \mathbb{R}^{d}$ be drawn from \(\Nc(0, d^{-1}I)\), and set $\tau = C_2 \frac{\sqrt{d}\log(d)\log(M)}{M}$. Then with probability at least $1 - 3\exp (-c d \log (d)) - 2\exp(-c'd\log(M))$ the
result of Theorem \ref{thm: main} holds uniformly for all \(x \in S^{d-1}\). 
\end{corollary}

\begin{proof}
As was the case in our previous lemmata that used a union bound argument and continuity, we will need to initially introduce some small modifications to our definitions to allow ourselves some wiggle room. As before, this will help us move from an argument that works for points on an $\epsilon$-net to arbitrary points. Consider for fixed \(x \in S^{d-1}\) the subset
\begin{align*}
    \hat{S}_{\tau, x} := \{i \in [M] : -\tau < \langle g^{(i)}, x \rangle < -\tau/2\}.
\end{align*}
Using a similar argument to that used in the proof of Lemma \ref{lem: S_tau}, we can show with high probability that \(C_6 d\log(d)\log(M) \leq |\hat{S}_{\tau, x}| \leq C_7 d\log(d)\log(M)\) for any fixed \(x\). Indeed, the expected value is \(\E|\hat{S}_{\tau, x}| = M \P\left(\sqrt{d^{-1}}g \in (-\tau, -\tau/2) \right)\) which means we can use upper and lower Riemann sum estimates to get analogous bounds as in Lemma \ref{lem: S_tau}. Lemma \ref{lem: feasible y, in the same half space UNIFORM} already holds uniformly so there's no need to use a union bound for that event. The arguments in Corollary \ref{cor: tail inner products positive, fixed y} and Lemma \ref{lem: smallest singular value Phi_y, uniform}  hold without any modifications because they concern the projected random variables \(P_{x^{\perp}}g^{(i)}\), the projections of  \(g^{(i)}\) on the orthogonal complement of the span of $x$. These random variables are independent of the event associated with \(\hat{S}_{\tau, x}\). Therefore, the statement of Theorem \ref{thm: main} also holds if \(S_{\tau, x}\) is replaced with \(\hat{S}_{\tau, x}\).

With that matter settled, condition on the event that \(\max_{i \in [M]} \|g^{(i)}\|_2 \leq C' \sqrt{\log(M)}\) which occurs with probability at least \(1 - \exp(-c' d \log(M))\) \cite{vershynin2018high}. Now, let \( \tau = \frac{C_2 \sqrt{d}\log(d)\log(M)}{M}\), \(\tau' = \frac{\tau}{2 C' \sqrt{\log(M)}},\) and let \(\Nc\) be a \(\tau'\) covering of \(S^{d-1}\). Note that \(|\Nc| \leq \left(\frac{3}{\tau'}\right)^d \leq M^{c''d}\). By a union bound argument, we have the result of Theorem \ref{thm: main} using \(\hat{S}_{\tau, y}\) in place of \(S_{\tau, y}\) holding for all \(y \in \Nc\) with probability at least \(1 - 3\exp\left(-cd\log(d)\right) - \exp(-c'd\log(M))\).

For an arbitrary \(x \in S^{d-1}\), let \(y\in \Nc\) satisfy \(\|x-y\|_2 \leq \tau'\). 
Then for any \(i \in \hat{S}_{\tau, y}\) we have
\begin{align*}
    \langle g^{(i)}, x \rangle &= \langle g^{(i)}, y \rangle + \langle g^{(i)}, x - y \rangle < -\tau/2 + C'\sqrt{\log(M)}\|x-y\|_2\\
    &< -\tau/2 + \tau/2 = 0,
\end{align*}
and, additionally, using a similar argument
\begin{align*}
    \langle g^{(i)}, x \rangle > -3\tau/2.
\end{align*}
In other words, we've just shown that \(\hat{S}_{\tau, y} \subset S_{3\tau/2, x}\), and furthermore that \(\sign( \langle g^{(i)}, x \rangle) = \sign( \langle g^{(i)}, y \rangle)\) for all \(i \in \hat{S}_{\tau, y}\). The inclusion \(\hat{S}_{\tau, y} \subset S_{3\tau/2, x}\) tells us that the cell containing \(x\) induced by the hyperplanes in \(S_{3\tau/2, x}\) has a radius no larger than the radius of the cell containing \(x\) induced by the hyperplanes from \(\hat{S}_{\tau, y}\). Denote the former cell by \(\Cc\) and the latter by \(\widehat{\Cc}\). By triangle inequality and Theorem \ref{thm: main}
\begin{align*}
    \max_{z \in \Cc} \|x-z\|_2 &\leq \max_{z \in \widehat{\Cc}} \|x-z\|_2 \leq \|x - y\|_2 + \max_{z \in\widehat{\Cc}} \|y - z\|\\ &\lesssim \frac{d\log(d)\log(M)}{\sqrt{M^2 + d^2 \log^2(d) \log^2(M)}}.
\end{align*}
\end{proof}
\begin{remark}
One potential practical application of this result is an encoding algorithm for efficiently representing a vector $x\in\R^d$ with $M$ one-bit measurements of the form $\sign(\langle g^{(i)},x \rangle)$. Since the cardinality of $S_{\tau,x}$ is $k \approx d~\mathrm{polylog}(M)$, one could store or transmit the indices associated with \(S_{\tau, x}\) using $\log_2{ {M} \choose {k}} \lesssim d~\mathrm{polylog}{(M)}$ bits in addition to the $k$ bits needed to encode the one bit measurements associated with \(S_{\tau, x}\). Using standard reconstruction techniques (see, e.g., \cite{powell2013quantization, goyal1995quantization}), one would then recover $x$ from $\sign(\langle g^{(i)},x\rangle),  ~i\in S_{\tau,x}$, with the error given in Corollary \ref{cor: uniform result}. This encoding scheme results in a root-exponential decay of the error in the number of bits used.
\end{remark}

\section{Acknowledgements}
The authors would like to thank the reviewers for their thoughtful and useful suggestions which have significantly improved the manuscript.

\section{Appendix}
\label{sec:app}
\begin{lemma}\label{lem: minimum singular value}
    Fix \(x=e_1\) and \(y\in S^{d-1}\). Let \(\tilde{G} \in \R^{|\SSS|\times (d-1)}\) be as in Lemma \ref{lem: smallest singular value Phi_y, uniform} and \(\Sigma := \E g^{(i)}_{[-1]} g^{(i)T}_{[-1]}\). Then for any \(t > 0\) we have
    \begin{align*}
        \left\| |\SSS|^{-1} \tilde{G}_y^T \tilde{G}_y - \Sigma\right\|_{op} &\leq K^2 \max\{\delta, \delta^2\}\\
        \delta :&= C\left( \sqrt{\frac{d-1}{|\SSS|}} + \frac{t}{\sqrt{|\SSS|}} \right),
    \end{align*}
    with probability at least \(1-2\exp(-t^2)\). Here, \(K := \max_{j \in \{2, \hdots, d}\} \|g^{(i)}_{j}\|_{\psi_2}\) where \(i \in \SSS\).
\end{lemma}
\begin{proof}
This proof is almost verbatim that of Theorem 4.6.1 in \cite{vershynin2018high} and we only include it here for completeness. Let \(\Nc \subset S^{d-2}\) be a \(1/4\)-net. By definition
\begin{align*}
    \left\| |\SSS|^{-1} \tilde{G}_y^T \tilde{G}_y - \Sigma\right\|_{op}^2 &= \sup_{\|z\|_2 = 1}\left\| |\SSS|^{-1} \|\tilde{G}_y z\|_2^2 - z^T \Sigma z\right\|_{2}^2\\
    &\leq 2 \sup_{z\in \Nc} \left\| |\SSS|^{-1} \|\tilde{G}_y z\|_2^2 - z^T \Sigma z\right\|_{2}^2.
\end{align*}
where the last line follows by Lemma 4.4.1 in \cite{vershynin2018high}. We will use a union bound to control this quantity. To that end, fix \(z \in S^{d-2}\). We can expand this random variable as
\begin{align*}
    \left\| |\SSS|^{-1} \|\tilde{G}_y z\|_2^2 - z^T \Sigma z\right\|_{2}^2 = |\SSS|^{-1} \sum_{i\in \SSS} z^T g^{(i)}_{[-1]} g^{(i)T}_{[-1]} z - z^T \Sigma z.
\end{align*}
The random variable \(g^{(i)T}_{[-1]} z\) is sub-gaussian with \(\|g^{(i)T}_{[-1]} z\|_{\psi_2} \leq \max_{j \in \{2, \hdots, d}\} \|g^{(i)}_{j}\|_{\psi_2} = K\) since the entries of \(g^{(i)}\) are independent. Therefore, the random variables \(z^T g^{(i)}_{[-1]} g^{(i)T}_{[-1]} z - z^T \Sigma z\) are sub-exponential with \(\|z^T g^{(i)}_{[-1]} g^{(i)T}_{[-1]} z - z^T \Sigma z\|_{\psi_1} \leq CK^2\). Using Bernstein's inequality \cite{vershynin2018high} we get
\begin{align*}
    \P\left( \left| |\SSS|^{-1} \sum_{i\in\SSS} z^T g^{(i)}_{[-1]} g^{(i)}_{[-1]} z - z^T \Sigma z \right| \geq \eps/2 \right)\\
    \leq 2\exp\left(-c \min\left\{ \frac{\eps^2}{K^4}, \frac{\eps}{K^2} \right\} |\SSS| \right).
\end{align*}
Setting \(\eps/K^2 = \max\{\delta, \delta^2\}\) and recalling the definition of \(\delta := C\left( \sqrt{\frac{d-1}{|\SSS|}} + \frac{t}{\sqrt{|\SSS|}} \right)\), this bound reduces to
\begin{align*}
     \P\left( \left| |\SSS|^{-1} \sum_{i\in\SSS} z^T g^{(i)}_{[-1]} g^{(i)T}_{[-1]} z - z^T \Sigma z \right| \geq \frac{1}{2}K^2 \max\{\delta, \delta^2\} \right)\\
     \leq 2\exp\left( -c C (d-1 + t^2)\right).
\end{align*}
To get a uniform bound over the net, we remark that \(|\Nc| \leq 9^{d-2}\), therefore
\begin{align*}
    \P\left( \max_{z\in \Nc} \left| |\SSS|^{-1} \sum_{i\in\SSS} z^T g^{(i)}_{[-1]} g^{(i)T}_{[-1]} z - z^T \Sigma z \right| \geq \frac{1}{2}K^2 \max\{\delta, \delta^2\} \right)\\
    \leq 9^{d-2} 2\exp\left( -c C (d-1 + t^2)\right) \leq 2\exp\left( -c t^2\right),
\end{align*}
provided \(C\) in the definition of \(\delta\) is at least \(\log(9)/c\).
\end{proof}

\bibliographystyle{IEEEtran}
\bibliography{bib}

\end{document}